\documentclass[a4paper,12pt]{article}

\usepackage{amsfonts,amssymb,amsmath}

\newtheorem{theo}{Theorem}[section]
\newtheorem{coro}[theo]{Corollary}
\newtheorem{lemm}[theo]{Lemma}
\newtheorem{prop}[theo]{Proposition}
\newtheorem{defi}[theo]{Definition}
\newtheorem{rema}[theo]{Remark}
\newtheorem{exam}[theo]{Example}

\newenvironment{proof}{\noindent \textbf{{Proof.}} \sf}

\def\qed{\hfill $\diamond$ \bigskip}

\def\lim{\mathop{\rm lim}\nolimits}

\def\HH{\mathsf H}

\def\Ext{\mathsf{Ext}}
\def\Hom{\mathsf{Hom}}

\def\Ker{\mathsf{Ker}}
\def\ker{\mathop{\rm Ker}\nolimits}
\def\coker{\mathsf{Coker}}

\def\Im{\mathsf{Im}}

\def\dim{\mathsf{dim}}

\begin{document}
\sf

\title{The first Hochschild (co)homology when adding arrows to a bound quiver algebra}
\author{Claude Cibils,  Marcelo Lanzilotta, Eduardo N. Marcos,\\ Sibylle Schroll and Andrea Solotar
\thanks{\footnotesize This work has been supported by the projects  UBACYT 20020130100533BA, PIP-CONICET 11220150100483CO, USP-COFECUB.
The third mentioned author was supported by the thematic project of FAPESP 2014/09310-5 and acknowledges support from the "Brazilian-French Network in Mathematics" and from a research scholarship from CNPq, Brazil. This work has been supported through the EPSRC  Early Career Fellowship EP/P016294/1 for the fourth author. The fifth mentioned author is a research member of CONICET (Argentina) and a Senior Associate at ICTP.}}

\date{}
\maketitle
\begin{abstract}
We provide a formula for the change of the dimension of the first Hoch\-schild cohomology vector space of bound quiver algebras when adding new arrows. For this purpose we show that there exists a short exact sequence which relates the first cohomology vector spaces of the algebras to the first relative cohomology. Moreover, we show that the first Hochschild homologies are isomorphic when adding new arrows.

\end{abstract}



\normalsize

\section{\sf Introduction}

The first Hochschild cohomology vector space $\HH\HH^1(B)$ of an algebra $B$ over a field $k$  is isomorphic to  the quotient of the $k$-derivations of the algebra by the inner ones.  It has a Lie algebra structure providing information on the algebra and it is an invariant of its derived equivalence class, see \cite{KELLER}.  As noticed for instance by M. Gerstenhaber in \cite[p.66]{GERSTENHABER1964}, derivations can be considered as infinitesimal automorphisms of $B$. They  are related to the deformation theory of $B$. R.-O. Buchweitz and S. Liu obtained in \cite{BUCHWEITZLIU} that if  $k$ is algebraically closed and $B$ is a finite dimensional algebra of finite representation type with $\Lambda$ its Auslander algebra, the following four statements are equivalent: $\HH\HH^1(B)=0$, $\HH\HH^1(\Lambda)=0$, $B$ is simply connected, and $\Lambda$ is strongly simply connected.

For a bound quiver algebra  $B=kQ/I$, given  some hypotheses on $I$, J. A. de la Pe\~{n}a and M. Saor\'\i n in \cite{DELAPENASAORIN} obtained formulas computing the dimension of $\HH\HH^1(B)$, see also \cite{CIBILS2000Cordoba,CIBILSSAORIN,CIBILSREDONDOSAORIN,HAPPEL}. For several families of algebras, results concerning the first cohomology vector space are given for instance in \cite{ASSEMBUSTAMANTEIGUSASCHIFFLER,ASSEMREDONDO,ASSEMREDONDOSCHIFFLER,LAUNOISLENAGAN,LEMEURa,LEMEURb,STRAMETZ,TAILLEFER}. In case the algebra $B$ is split, a canonical decomposition  of $\HH\HH^1(B)$ into four direct summands is obtained in \cite{CMRS2002}.

In this paper we study the change in both the Hochschild cohomology and Hochschild homology of an algebra given by quiver and relations, when we add arrows to the quiver. More precisely,  we consider a bound quiver algebra $B=kQ/I$ and a finite set of new arrows $F$ that we add to $Q$. The new quiver is denoted by $Q_F$ and $B_F$ denotes the corresponding algebra, that  is  $B_F=kQ_F/\langle I\rangle_{kQ_F}$ where the denominator is the two sided ideal generated by $I$ in $kQ_F$. A relative path is a sequence $(a_n,\dots ,a_1)$ of new arrows such that $s(a_{i+1})Bt(a_i)\neq 0$ for $i=1,\dots n-1$, where $s,t:F\to Q_0$ are the source and the target maps of the new arrows. We observe that $B_F$ is  finite dimensional if and only if there are no relative cycles as defined in Section \ref{relative path}. In this case, namely if $B_F$ is a bound quiver algebra, we obtain a formula for computing $\Delta = \dim_k\HH\HH^1(B_F)-\dim_k\HH\HH^1(B)$, see Theorem \ref{formulaH^1}. Next we specialize the formula to the case
where adding only one arrow, see Corollary \ref{onearrow}.

Note that in a first step, instead of adding new arrows, it is possible to simply add new vertices. New arrows can then also be attached to these new vertices. Indeed, adding new vertices does not change Hochschild cohomology except in degree zero, that is the dimension of the center increases by the number of the new vertices.

The procedure of adding arrows without changing the relations (or the  reverse procedure, namely deleting arrows which are not involved in a minimal set of generators of the relations) has been also recently considered in \cite{GREENPSAROUDAKISSOLBERG} in relation to the finitistic dimension. This procedure is also used in  \cite{CIBILSLANZILOTTAMARCOSSOLOTAR} in order to compute the change of dimensions in Hochschild cohomology and homology in degrees greater or equal to two. The formulas provided here for the first Hochschild cohomology vector space are however more intricate.

A main tool for our work is relative cohomology as defined by G. Hochschild in \cite{HOCHSCHILD1956} and used for instance in \cite{AUSLANDERSOLBERG1,SOLBERG} in the context of representation theory. We prove the existence of a short exact sequence which relates the relative and the usual  cohomologies in degree one,  see Proposition~\ref{beginning exact}.  Our proof uses the fact that  $B_F$ is  a tensor algebra of a projective $B$-bimodule over $B$. Indeed, in \cite{CIBILSLANZILOTTAMARCOSSOLOTAR} it is proven that a tensor algebra has a relative projective resolution of length one,  which enables us to compute in the present paper the dimensions of the extremities of the short exact sequence.   We also note that the relative projective resolution of length one, together with the Jacobi-Zariski long exact sequence obtained by A. Kaygun in \cite{KAYGUN,KAYGUNe} provide another proof of the existence of the short exact sequence in Proposition~\ref{beginning exact}.

As a by-product of our formula, we obtain a new computation of the dimension of $\HH\HH^1(kQ)$ for a quiver $Q$ without cycles.

In the last section we dualize the short exact sequence, and we show that the first Hochschild homology vector space does  not change when adding new arrows without relative cycles.

 We use the symbol $\simeq$ when there exists an isomorphism, while $=$  means either equality or canonical isomorphism.

\section{\sf Adding new arrows}\label{Adding new arrows}

A  quiver $Q$ consists of two  sets, the set of vertices $Q_0$, the set of arrows $Q_1$, and two maps $s,t: Q_1\to Q_0$ called respectively the source and the target maps. In this paper we will only consider finite quivers, that is $Q_0$ and $Q_1$ are finite. A \emph{path} of length $n>0$ is a sequence of arrows $ \gamma=(a_n,\dots, a_1)$ such that $s(a_{i+1})=t(a_i)$ for $i=1,\dots,n-1$, and we put $t(\gamma)=t(a_n)$ and $s(\gamma)=s(a_1)$. Let $Q_n$ be the set of paths of length $n$. The vertices are the paths of length $0$, each one is its own source and target. Let $Q_*=\bigcup_{n\geq 0}Q_n$. The path algebra $kQ$ is the vector space with basis $Q_*$,  the product  of two paths $\beta$ and $\alpha$ is their concatenation $\beta\alpha$ if $t(\alpha)=s(\beta)$, and $0$ otherwise. The vertices are a complete set of orthogonal idempotents.

If $k$ is algebraically closed, by a result of P. Gabriel in \cite{GABRIEL1973,GABRIEL1980}, see also for instance
 \cite[Theorem 3.7]{ASSEMSIMSONSKOWRONSKY} or \cite{SCHIFFLER}, any finite  dimensional $k$-algebra $B$ is Morita equivalent to an algebra $kQ/I$, where $Q$ is determined by $B$ and $I$ is an admissible two-sided ideal of $kQ$, that is there exists $n\geq 2$ such that $\langle Q_n\rangle \subset I \subset \langle Q_2 \rangle$. Such an algebra is finite dimensional and is called a \emph{bound quiver algebra}.

Next we introduce some definitions and notations.

  \label{B_F}
  Let $Q$ be a quiver. A \emph{set of new arrows }is a finite set $F$ with two map $s,t: F\to Q_0$. The quiver $Q_F$ is given by $(Q_F)_0=Q_0$ and $(Q_F)_1= Q_1\sqcup F$, while $s$ and $t$ are inferred from the corresponding maps of $Q_1$ and $F$.  Let now $B=kQ/I$ be a bound quiver algebra and let $\langle I\rangle_{kQ_F}$ be the two-sided ideal generated by $I$ in $kQ_F$. We denote $B_F$ the algebra $kQ_F/\langle I\rangle_{kQ_F}$.

  \label{relative path} A \emph{relative path} of length $n>0$ is a sequence of new arrows $\gamma =(a_n,\dots, a_1)$ such that $s(a_{i+1})Bt(a_i)\neq 0 \mbox{ for } i=1,\dots,n-1$. We put $s(\gamma)=s(a_1)$ and $t(\gamma)=t(a_n)$. The \emph{dimension} of $\gamma$ is $\dim_k \gamma= \prod_{i=1}^{n-1} \dim_k s(a_{i+1})Bt(a_i).$

   The set of relative paths of length $n$ is denoted by $R_n$. We set $R_*=\bigcup_{n>0}R_n$.

 A relative path $\gamma$ is called a \emph{relative cycle} if  $s(\gamma)Bt(\gamma)\neq 0$. Its \emph{cyclic dimension}  is $\mathsf{cdim}_k\gamma = \dim_k(s(\gamma)Bt(\gamma))\dim_k \gamma$. Note that if $a\in F$  is such that $s(a)Bt(a)\neq 0$, then $a$ is a relative cycle. We also call $a$ a relative loop.

   For instance let $$Q=f\cdot \rightarrow \cdot e$$  and let $a$ be a new arrow in the reverse direction. We have $$Q_F= f\cdot \leftrightarrows \cdot e$$ and $a$ is a relative loop.

  To $F$ we associate  the projective $B$-bimodule
 $$N=\bigoplus_{a\in F} Bt(a)\otimes s(a)B.$$

  \normalsize

The following has been proved in \cite{CIBILSLANZILOTTAMARCOSSOLOTAR}.
\begin{prop}\label{isomorphic}
Let $B=kQ/I$ be a bound quiver algebra and let $F$ be a set of new arrows. The algebras $B_F$ and
$$T_B(N)=B\ \oplus\ N\ \oplus\ N\otimes_B N\ \oplus \ N\otimes_B N\otimes_B N\ \oplus \ \cdots$$
are canonically isomorphic.
\end{prop}
Through the isomorphism of Proposition \ref{isomorphic}, a new arrow $a$ corresponds to $t(a)\otimes s(a)\in N$. Moreover, for $n>0$ we have that
$$N^{\otimes_Bn}\simeq \bigoplus_{\gamma \in R_n} \dim_k\gamma\ (Bt(\gamma)\otimes s(\gamma)B)$$
 corresponds to
$$\bigoplus_{(a_n,\dots a_1)\in R_n} Ba_nB\dots a_1B.$$
We infer the following
\begin{coro}
The algebra $B_F$ is of finite dimension if and only if there are no relative cycles.
\end{coro}

\section{\sf Short exact sequence}\label{Short exact sequence}

 In this section we establish a short exact sequence which relates the ordinary first Hochschild cohomology with the relative one, when adding new arrows to the quiver of a bound quiver algebra.

Let $B\subset A$ be an inclusion of $k$-algebras, and let $X$ be an $A$-bimodule. The relative Hochschild cohomology vector spaces for $n\geq 0$ -- see \cite{HOCHSCHILD1956} -- are
$$\HH^n(A|B, X)=\Ext_{A\otimes A^{\mathsf{op}}|B\otimes  B^{\mathsf{op}}}^n(A,X),$$
where the latter are the relative $\Ext$ groups for the inclusion of algebras $$B\otimes B^{\mathsf{op}}\subset A\otimes A^{\mathsf{op}}.$$

\begin{rema}   We have that
$$\Ext_{A\otimes A^{\mathsf{op}}|B\otimes B^{\mathsf{op}}}^n(A,X)= \Ext_{A\otimes A^{\mathsf{op}}|B\otimes A^{\mathsf{op}}}^n(A,X).$$
Indeed the relative bar resolution of $A$ by $B\otimes B^{\mathsf{op}}$-relative projective modules
$$ \cdots \to A\otimes_B A\otimes_B A \otimes_B A \to A\otimes_B A\otimes_B A \to A\otimes_B A \to A$$
 admits a contraction of homotopy of $B\otimes A^{\mathsf{op}}$-modules. Since $B\otimes B^{\mathsf{op}}\subset B\otimes A^{\mathsf{op}}$, the modules involved are also $B\otimes A^{\mathsf{op}}$-relative projective. Hence the above is a $B\otimes A^{\mathsf{op}}$-relative projective resolution which can be used for computing the corresponding relative Ext.
\end{rema}

\begin{lemm}\label{derivation bimodule morphism}
Let $\varphi:A\to X$ be a derivation. The map $\varphi$ is a $B$-bimodule morphism if and only if $\left.\varphi\right|_B=0.$
\end{lemm}
\begin{proof}
If $\varphi$ is a $B$-bimodule morphism, then for any $b\in B$ we have $\varphi(b)=b\varphi(1)$. Moreover, since $\varphi$ is a derivation we have $\varphi(b)=b\varphi(1)+\varphi(b)1,$
hence $\varphi(b)=0.$ Conversely for all $a\in A$ we have
$\varphi(ba)=b\varphi(a)+\varphi(b)a=b\varphi(a).$
Similarly $\varphi(ab)=\varphi(a)b$. \qed
\end{proof}

A. Kaygun in \cite{KAYGUN, KAYGUNe} has obtained a Jacobi-Zariski long exact sequence whenever $A/B$ is projective as a $B$-bimodule, and $X$ is finite dimensional (actually the hypotheses are slightly more general). The next result shows that the beginning of the sequence is always exact for any inclusion of algebras $B\subset A$ and any $A$-bimodule $X$.

\begin{prop}\label{beginning exact}
The following sequence is exact
$$0\to \HH^1(A|B, X)\stackrel{\iota}{\to} \HH^1(A, X) \stackrel{\kappa}{\to} \HH^1(B, X)$$
\end{prop}
\begin{proof}
Let $\varphi:A\to X$ be a relative derivation. Namely $\varphi$ is a derivation which is a $B$-bimodule map. Suppose that $\iota(\varphi)$ is $0$. In other words there exists $x\in X$ such that $\iota(\varphi)(a)=xa-ax.$
By Lemma \ref{derivation bimodule morphism}, we have $\left.\varphi\right|_B=\left.\iota(\varphi)\right|_B=0$. Hence $$x\in X^B=\{x\in X\mid  bx=xb \mbox{ for all } b\in B\}.$$
That is  $\iota(\varphi)$ is a relative inner derivation, and therefore is $0$ in $\HH^1(A|B, X)$.

Note that by Lemma \ref{derivation bimodule morphism} we already have $\Im\iota\subset \Ker\kappa$. In order to prove the reverse inclusion, let $\varphi:A\to X$ be a derivation such that $\left.\varphi\right|_B:B\to X$ is inner. That is there exists $x\in X$ such that for all $b\in B$ we have $\varphi(b)=xb-bx.$ Let $\varphi_x:A\to X$ be the inner derivation given by $\varphi_x(a)=xa-ax$. Hence $\varphi$ and $\varphi - \varphi_x$ are equal in $\HH^1(A, X)$. Moreover $\left.(\varphi-\varphi_x)\right|_B=0$. By Lemma \ref{derivation bimodule morphism} the derivation $\varphi-\varphi_x$ is a $B$-bimodule map, and therefore $\varphi-\varphi_x \in \Im\iota$.\qed
\end{proof}

\begin{rema}
The first Hochschild cohomology vector space of an algebra $A$ is a Lie algebra: if $\varphi, \psi :A\to A$ are derivations, then $[\varphi,\psi]=\varphi\psi-\psi\varphi$. Moreover $\HH^1(A|B, A)$ is also a Lie algebra by a straightforward verification, and $\iota:\HH^1(A|B, A)\to \HH\HH^1(A)$ is a morphism of Lie algebras. The question naturally arises to know if the Lie subalgebra $\HH^1(A|B, A)$ of $\HH\HH^1(A)$ has an ideal complementing it. The following example shows that this is not the case.
\end{rema}

\begin{exam}
Let $Q=e\cdot\rightrightarrows \cdot f $ be the Kronecker quiver with arrows $u$ and $v$, and let $B=kQ$. Let $a$ be a new arrow from $e$ to $f$, and let $B_F=kQ_F$.  It is straightforward to show that $\HH\HH^1(B_F)\simeq \mathsf{sl}_3(k)$. This Lie algebra is simple, hence there is no ideal complementing the proper non zero sub-Lie algebra $\HH^1(B_F|B, B_F)$.
\end{exam}

\begin{theo}\label{theshortexact}
Let $B=kQ/I$ be a bound quiver algebra, let $F$ be a set of new arrows with no relative cycles, and let $B_F$ be the algebra defined in Section \ref{B_F}. Let $X$ be a $B_F$-bimodule. There is a short exact sequence
$$0\to \HH^1(B_F|B, X)\stackrel{\iota}{\to} \HH^1(B_F, X) \stackrel{\kappa}{\to} \HH^1(B, X)\to 0.$$
\end{theo}
\begin{proof}
To prove that $\kappa$ is surjective, let $\varphi: B\to X$ be a derivation. Let $\gamma=(a_n,\dots a_1)$ be a relative path and let
$$ Ba_nB\dots Ba_1B = Bt(a_n)\otimes s(a_n)B\dots Bt(a_1)\otimes s(a_1)B $$
be the corresponding direct summand in $B_F=T_B(N)$. Let $\varphi': B_F\to X$ be given by
$$\varphi'(\beta_{n+1}a_n\beta_n\dots\beta_2a_1\beta_1)=\sum_{i=1}^{n+1} \beta_{n+1}a_n\beta_n\dots\varphi(\beta_i)\dots\beta_2a_1\beta_1,$$
 $\left.\varphi'\right|_B=\varphi$. Observe that $\varphi'(a)=0$ for all $a\in F$. It is straightforward to verify that $\varphi'$ is a derivation. \qed
\end{proof}
As mentioned before, A. Kaygun has obtained a Jacobi-Zariski long cohomological exact sequence if $A/B$ is a projective $B$-bimodule, and if $X$ is finite dimensional. In our previous result, $B_F/B$ is a projective $B$-bimodule. Moreover $\HH^2(B_F|B, X)=0,$ as a consequence of Theorem \cite[Theorem 3.3]{CIBILSLANZILOTTAMARCOSSOLOTAR}. Hence if $X$ is finite dimensional the above short exact sequence can be inferred from \cite{KAYGUN,KAYGUNe} and from \cite{CIBILSLANZILOTTAMARCOSSOLOTAR}.

\section{\sf First Hochschild cohomology}\label{First Hochschild cohomology}

Let $B=kQ/I$ be a bound quiver algebra, and let $F$ be a set of new arrows with no relative cycles, that is $B_F$ is finite dimensional.  In this section we provide a formula which computes the dimension of $\HH\HH^1(B_F)$.
\begin{defi}\label{extended}
A \emph{extended relative path} $\omega$ of length $n>0$ is a sequence $(y,\gamma, x)$ where $y,x\in Q_0$  and $\gamma$ is a relative path of length $n$ such that $$yBt(\gamma)\neq 0 \mbox{ and } s(\gamma)Bx\neq 0.$$
We set $t(\omega)=y$ and $s(\omega)=x$. Moreover we put
$$\dim_k\omega = \left(\dim_k yBt(\gamma)\right)(\dim_k\gamma) \left(\dim_k s(\gamma)Bx\right).$$
The set of extended relative paths of length $n$ is denoted by $W_n$.

An \emph{extended relative path} of length $0$ is $\omega=(y,x)$ such that $yBx\neq 0$. In this case, we set $\dim_k \omega =\dim_kyBx$. The set of extended paths of length $0$ is denoted by $W_0$, while $W_*=\bigcup_{n\geq 0}W_n.$
\end{defi}
For the purpose of this section, let $ZA$ denote the center of an algebra $A$. We are going to use $I_x$ and $P_x$ respectively for the injective envelope and the projective cover of the simple $B$-module at $x$ for $x\in Q_0$. And finally, we use $F/\!/W_*$ for the set of pairs $(a,\omega)\in F\times W_*$ such that $s(a)=s(\omega)$ and $t(a)=t(\omega).$

\begin{theo}\label{formulaH^1}
Let $B=kQ/I$ be a bound quiver algebra, and let $F$ be a set of new arrows with no relative cycles.
Let $\Delta= \dim_k\HH\HH^1(B_F)- \dim_k\HH\HH^1(B)$. Then
\begin{align*}
\Delta =\ &\dim_k ZB_F-\dim_k ZB \ + \sum_{(a,\omega)\in F/\!/W_*} \dim_k\omega \ +\\
&\sum_{\gamma\in R_*} \dim_k\gamma\left( \dim_k\Ext^1_B(I_{s(\gamma)}, P_{t(\gamma)}) -  \dim_k\Hom_B(I_{s(\gamma)}, P_{t(\gamma)})\right).
\end{align*}
\end{theo}
Before proving this formula, we state two corollaries and we consider two examples.
\begin{coro}\label{onearrow}
Let $a$ be a single new arrow from $e$ to $f$ which is not a relative loop. Let $\Delta= \dim_k\HH\HH^1(B_{\{a\}})- \dim_k\HH\HH^1(B)$. Then
\begin{align*} \Delta = &\ \dim_k ZB_{\{a\}}-\dim_k ZB \\&+ \dim_kfBe + \dim_k fBf \ \dim_k eBe \\
&+ \dim_k\Ext^1_B(I_{e}, P_{f}) -  \dim_k\Hom_B(I_{e}, P_{f}).
\end{align*}\end{coro}
\begin{proof}
We have that $\{a\}/\!/W_*=\{(a,(f,e)), (a,(f,a,e))\}$ because $a$ is not a relative loop.
Hence
\begin{align*}\sum_{(a,\omega)\in F/\!/W_*} \dim_k\omega=&\dim_k(f,e)+\dim_k(f,a,e)=\\&\dim_kfBe + \dim_k fBf \ \dim_k eBe.
\end{align*}
On the other hand $R_*=\{(a)\}.$ \qed
\end{proof}

\begin{exam}
Let $Q$ be the quiver $$f \cdot \longrightarrow\cdot\longrightarrow\cdots\longrightarrow\cdot\longrightarrow\cdot e$$
and let $\beta_5\beta_4\beta_3\beta_2\beta_1$ be a decomposition of the path from $f$ to $e$, where the lengths  $l(\beta_4), l(\beta_3)$ and $l(\beta_2)$ are strictly positive. Hence there are at least $4$ vertices in $Q$. Let $I=\langle \beta_4\beta_3, \beta_3\beta_2\rangle$, and let $B=kQ/I$. There are no cycles of positive length and $Q$ is connected, hence $ZB=k$. Let $E=kQ_0$. Since $E$ is semisimple, $\HH\HH^1(B)= \HH^1(B|E, B$), and an easy computation shows that the latter is $0$.

Let $F=\{a\}$ be a new arrow from $e$ to $f$. Observe that $a$ is not a relative loop, that is $B_F$ is finite dimensional. The non-zero cycles of $B_F$ are the cycles whose sources are at the vertices of $\beta_3$ different from $s(\beta_3)$ and  $t(\beta_3)$. If $l(\beta_3)=1$,  this set is empty.  The sum of the non zero cycles is an element of the center of $B_F$. Hence
\begin{itemize}
\item if $l(\beta_3)>1$ then $\dim_kZB_F=2$,
\item if $l(\beta_3)=1$ then $\dim_kZB_F=1$.
\end{itemize}
Let $y$ be the target of the first arrow of $\beta_3$, we have $I_e=P_{y}$ and $\Ext^1_B(I_e, P_f)=0$. On the other hand $\dim_k\Hom_B(I_e, P_f)=\dim_k\Hom_B(P_y, P_f)=\dim_k yBf$. Hence
\begin{itemize}
\item if $l(\beta_3)>1$ then $\dim_k\Hom_B(I_e, P_f)=1$,
\item if $l(\beta_3)=1$ then $y=t(\beta_3)$ and $\dim_k\Hom_B(I_e, P_f)=0$.
\end{itemize}
Hence by the previous corollary the dimension of $\HH\HH^1(B_{\{a\}})$ is always equal to one.
\end{exam}

\begin{exam}
A \emph{toupie quiver} $Q$ has a single source vertex $e$, a single sink vertex $f$, and other vertices are the start of exactly one arrow, as well as the target of exactly one arrow. A path from $e$ to $f$ is called a \emph{branch}. The bound quiver algebra $B=kQ/I$ is called a \emph{toupie algebra}. Note that canonical algebras introduced in \cite{RINGEL} are instances of toupie algebras. The dimension of the first Hochschild cohomology of toupie algebras has been computed as follows in \cite{GATICALANZILOTTA}, see also \cite{ARTENSTEINLANZILOTTASOLOTAR}.

Let $a$ be the number of branches which are arrows. Let $m$ be the number of branches which are in $I$. Furthermore, two branches are related if they appear in the same minimal relation of $I$. This generates an  equivalence relation among these branches. Denoting by $r$ the number of equivalence classes we have $\dim_k\HH\HH^1(B) = r+m+a(\dim_kfBe) -1.$

We add a new branch of length $n\geq 2$ from $e$ to $f$. This is obtained by adding first $n-1$ new vertices providing the quiver $Q'$ and the algebra $B'=kQ'/I'$ where $I'=\langle I\rangle_{kQ'}$. We have that $\dim_k ZB'=n$, while $\HH\HH^1(B)=\HH\HH^1(B')$. Then consider the appropriate set $F$ of $n$ new arrows to get the new branch. The algebra $B'_F$ is still a toupie algebra, and a simple computation shows that the formula above is in accordance with  Theorem \ref{formulaH^1}.

Adding a new arrow from $e$ to $f$ also provides a toupie algebra, and we also have an accordance between the formulas.

\end{exam}

 Using Theorem \ref{formulaH^1} we get a new proof of the following result.
\begin{coro}\cite{CIBILS2000Cordoba}
Let $Q$ be a quiver without cycles. Let $c$ be the number of connected components of $Q$, and $Q_*$ the set of paths.
$$\dim_k\HH\HH^1(kQ)=c-|Q_0|+|Q_1/\!/Q_*|.$$
\end{coro}
\begin{proof}
Let $B=kQ_0$ and let $F=Q_1$, hence $B_F=kQ$. Since $B$ is semisimple, $\HH\HH^1(B)=0$. Moreover $B$ is commutative, hence $\dim_kZB=|Q_0|$. On the other hand $\dim_kZ(kQ)=c$, indeed each element of a basis of the center is the sum of the vertices of a connected component.

The relative paths are the paths of positive length, each one has dimension $1$. Enlarged relative paths are all the paths, their dimension is also $1$.

Since $B$ is semisimple, $\Ext^1_B$ vanishes. Moreover, for any $\gamma \in R_*$ we have $s(\gamma)\neq t(\gamma)$. Note that for $x\in Q_0$ we have $I_x=P_x=S_x$ where $S_x$ is the simple module at $x$. Hence
$$\Hom_B(I_{s(\gamma)}, P_{t(\gamma)})=\Hom_B(S_{s(\gamma)}, S_{t(\gamma)})=0.$$
Finally $F/\!/ W_*= Q_1/\!/ Q_*.$\qed
\end{proof}

The proof of Theorem \ref{formulaH^1} relies on the short exact sequence of Theorem \ref{theshortexact} and on the following results.

\begin{lemm}\label{H^1(B,B_F)}In the set up of Theorem \ref{formulaH^1} there is a decomposition
$$H^1(B,B_F)\ \simeq\ \HH\HH^1(B)\oplus\bigoplus_{\gamma\in R_*}\dim_k\gamma \ \Ext^1_B(I_{s(\gamma)}, P_{t(\gamma)}).$$
\end{lemm}
\begin{proof}
We have that
$$B_F=T_B(N)=B\oplus\ N\ \oplus\ N\otimes_B N\ \oplus \ \cdots$$
is a $B$-bimodule decomposition. For $n>0$ there is an isomorphism of $B$-bimodules
$$N^{\otimes_Bn}\simeq \bigoplus_{\gamma \in R_n} \dim_k\gamma \ Bt(\gamma)\otimes s(\gamma)B.$$

Let $U$ and $V$ be respectively a left and a right $B$-module of finite dimension over $k$, and let $V'=\Hom_k(V,k)$ be the $k$-linear dual of $V$. There is a canonical isomorphism of $B$-bimodules between $U\otimes V$ and $\Hom_k(V',U)$. Hence
$$Bt(\gamma)\otimes s(\gamma)B= \Hom_k((s(\gamma)B)', Bt(\gamma)).$$
Moreover $Bt(\gamma)= P_{t(\gamma)}$ and $(s(\gamma)B)'=I_{s(\gamma)}$. Finally it is well known (see  \cite[p. 170, 4.4]{CARTANEILENBERG}) that for left $B$-modules $Y$ and $Z$ there is a canonical isomorphism of vector spaces between $H^n(B, \Hom_k(Y,Z))$ and $\Ext_B^n(Y,Z)$. \qed
\end{proof}

Next we will compute the dimension of the right term of the short exact sequence in Theorem \ref{theshortexact}. For this purpose we recall the following result and we infer a consequence.

\begin{theo}\cite{CIBILSLANZILOTTAMARCOSSOLOTAR}\label{relative resolution}
Let $B$ be a $k$-algebra, let $M$ be any $B$-bimodule and let $T=T_B(M)$ be the tensor algebra. There is a $T\otimes T^{\mathsf{op}}|B\otimes T^{\mathsf{op}}$   projective resolution of $T$
$$0\longrightarrow T\otimes_B M\otimes_B T \stackrel{d}{\longrightarrow} T\otimes_B T\longrightarrow T\longrightarrow 0.$$
\end{theo}
 \begin{rema}
The $T$-bimodules of the resolution above are $B\otimes B^{\mathsf{op}}$-relative projective, hence they are also $B\otimes T^{\mathsf{op}}$-relative projective. In \cite{CIBILSLANZILOTTAMARCOSSOLOTAR} it is proven that there exists a $B\otimes T^{\mathsf{op}}$ contraction of homotopy.
\end{rema}
\begin{coro}\label{A}
Let $X$ be a $B$-bimodule and let $\delta: X^B\to \Hom_{B-B}(M,T)$ be the linear map obtained from $d$  by applying the functor $\Hom_{T-T}(\ \ ,X)$ followed by the canonical identification. Then
$$\HH^1(T|B, X)\simeq \coker \delta.$$
In case $T$ and $X$ are finite dimensional, we have
$$\dim_k\HH^1(T|B,X)=\dim_k\Hom_{B-B}(M,T)-\dim_kX^B+\dim_kX^T.$$
\end{coro}

\begin{proof}
The resolution of the previous theorem provides the isomorphism. Hence, in the finite dimensional case we have
$$\dim_k \HH^1(T|B, X)= \dim_k \Hom_{B-B}(M,T) - \dim_k\Im \delta,$$
while
$$\dim_k\Im\delta= \dim_kX^B-\dim_k\Ker \delta.$$
Moreover $\Ker\delta =\HH^0(T|B,X).$ It is well known and easy to prove that for an inclusion of algebras $B\subset A$ and an $A$-bimodule $X$ we have
$$\HH^0(A|B,X)=\HH^0(A,X)=X^A.$$\qed
\end{proof}
\begin{prop}\label{B}
Let $kQ/I$ be a bound quiver algebra, let $F$ be a set of new arrows with no relative cycles, let $N$ be the $B$-bimodule associated to $F$ and let $B_F=kQ_F/\langle I\rangle_{kQ_F}= T_B(N).$ Then
$$\dim_k\Hom_{B-B}(N, B_F)= \sum_{(a,\omega)\in F/\!/ W_*}\dim_k\omega.$$
\end{prop}
\begin{proof}
Let $e,f\in Q_0$, let $Y$ be a $B$-bimodule and let $Bf\otimes eB$ be the $B$-bimodule corresponding to the idempotent $f\otimes e\in B\otimes B^{\mathsf{op}}$. Recall that
$$\Hom_{B-B}(Bf\otimes eB,Y)=fYe.$$
Hence for $n>0$
\begin{align*}
\Hom_{B-B}(N,N^{\otimes_B n})=&\bigoplus_{a\in F} t(a)N^{\otimes_B n} s(a)\\
   \simeq &\bigoplus_{a\in F}\bigoplus_{\gamma \in R_n}\dim_k\gamma\ \ t(a)Bt(\gamma)\otimes s(\gamma)Bs(a).
\end{align*}
Then
$$\dim_k \Hom_{B-B}(N,N^{\otimes_B n})=\sum_{(a,\omega)\in F/\!/W_n} \dim_k\omega.$$
For $n=0$, we have
$$\Hom_{B-B}(N,B)=\bigoplus_{a\in F} t(a)Bs(a).$$
Note that $(t(a), s(a))$ is an extended relative path of length $0$ if $t(a)Bs(a)\neq0.$ We also have
$$\dim_k \Hom_{B-B}(N,B)=\sum_{(a,\omega)\in F/\!/W_0} \dim_k\omega.$$\qed
\end{proof}

We now give the  proof of Theorem \ref{formulaH^1}.

\begin{proof}
Consider the short exact sequence of Theorem \ref{theshortexact} for the $B_F$-bimodule $B_F$
$$0\to \HH^1(B_F|B, B_F)\stackrel{\iota}{\to} \HH\HH^1(B_F) \stackrel{\kappa}{\to} \HH^1(B, B_F)\to 0.$$
The dimension of the right term is given by Lemma \ref{H^1(B,B_F)}. The dimension of the left term is obtained through Corollary \ref{A} and Proposition \ref{B}.\qed
\end{proof}

\section{\sf First Hochschild homology}\label{First Hochschild homology}
In this section we will prove that adding new arrows without creating relative cycles does not change the first Hochschild homology.
\begin{theo}\label{seshomology}
Let $B=kQ/I$ be a bound quiver algebra and let $F$ be a set of new arrows without relative cycles. Let $Y$ be a finite dimensional $B_F$-bimodule. There is an exact sequence
$$0\to \HH_1(B,Y)\to \HH_1(B_F,Y)\to \HH_1(B_F|B,Y)\to 0.$$
\end{theo}
\begin{proof}By taking  the dual of the short exact sequence of Theorem \ref{theshortexact} for $X$ a finite dimensional module, we obtain the short exact sequence
$$0\to  \HH^1(B, X)' {\to} \HH^1(B_F, X)' {\to} \HH^1(B_F|B, X)'\to 0.$$
It is well known and easy to prove that for a finite dimensional algebra $A$ and a finite dimensional $A$-bimodule $Z$ we have for $n\geq 0$
$$\HH^n (A,Z)'=\HH_n(A,Z').$$
Moreover, this also holds for relative Hochschild (co)homology, the proof is along the same lines using the relative bar resolution and the resulting complexes of (co)chains inferred in \cite{HOCHSCHILD1956}. Hence there is a short exact sequence
$$0\to \HH_1(B,X')\to \HH_1(B_F,X')\to \HH_1(B_F|B,X')\to 0$$
and we set $Y=X'$.\qed
\end{proof}
\begin{rema}
The short exact sequence above cannot be inferred from the long exact sequence of A. Kaygun for homology obtained in \cite{KAYGUN,KAYGUNe}, since the relative resolution of $T_B(N)=B_F$ of Theorem \ref{relative resolution}  only provides $H_2(B_F|B, Y)=0$.
\end{rema}
Recall that for an algebra $A$ and an $A$-bimodule $Y$, the coinvariants are
$$Y_A=A\otimes_{A-A}Y= Y/\langle ay-ya \mid a \in A, y \in Y \rangle= \HH_0(A,Y).$$
\begin{lemm}
Let $B=kQ/I$ be a bound quiver algebra and let $F$ be a set of new arrows without relative cycles.  Then
$$\HH_1(B_F|B,B_F)=0.$$
\end{lemm}
\begin{proof}
The relative resolution of Theorem \ref{relative resolution} gives
$$\HH_1(B_F|B,B_F)\simeq \ker \left(B_F\otimes_{B-B}N\longrightarrow (B_F)_B\right).$$
Let $RC_n$ be the set of relative cycles of length $n$. We have
$$\dim_k \left(N^{\otimes_B n}\otimes_{B-B}N\right)=\sum_{\gamma\in RC_{n+1}}\mathsf{cdim}_k\gamma$$
$$\dim_k \left(B\otimes_{B-B}N\right)=\sum_{\gamma\in RC_1}\mathsf{cdim}_k\gamma.$$
Since for all $n$ the set of relative cycles $RC_n$ is empty, we deduce that these vector spaces vanish. Hence  $B_F\otimes_{B-B}N=0$ and the result follows.\qed
\end{proof}

\begin{lemm} Let $B=kQ/I$ be a bound quiver algebra and let $F$ be a set of new arrows without relative cycles.  Then
$$\HH_1(B,B_F)=\HH\HH_1(B).$$
\end{lemm}
\begin{proof}
For $n>0$ we have
$$\HH_1(B, N^{\otimes_B n}) \simeq \bigoplus_{\gamma\in R_n}\dim_k\gamma\ \HH_1(B, Bt(\gamma)\otimes s(\gamma)B).$$
The $B$-bimodule $Bt(\gamma)\otimes s(\gamma)B$ is projective, hence its homology in positive degrees is zero.\qed
\end{proof}
The following is a direct consequence of the short exact sequence of Theorem \ref{seshomology} and of the previous lemmas.
\begin{theo}
Let $B=kQ/I$ be a bound quiver algebra, and let $F$ be a set of new arrows with no relative cycles. Then
$$\HH\HH_1(B_F)=\HH\HH_1(B).$$
\end{theo}

\footnotesize
\noindent C.C.:\\
Institut Montpelli\'{e}rain Alexander Grothendieck, CNRS, Univ. Montpellier, France.\\
{\tt Claude.Cibils@umontpellier.fr}

\medskip
\noindent M.L.:\\
Instituto de Matem\'atica y Estad\'\i stica  ``Rafael Laguardia'', Facultad de Ingenier\'\i a, Universidad de la Rep\'ublica, Uruguay.\\
{\tt marclan@fing.edu.uy}

\medskip
\noindent E.N.M.:\\
Departamento de Matem\'atica, IME-USP, Universidade de S\~ao Paulo, Brazil.\\
{\tt enmarcos@ime.usp.br}

\medskip
\noindent S.S.:\\
Department of Mathematics, University of Leicester, United Kingdom.\\
{\tt schroll@leicester.ac.uk}

\medskip
\noindent A.S.:
\\IMAS-CONICET y Departamento de Matem\'atica,
 Facultad de Ciencias Exactas y Naturales,\\
 Universidad de Buenos Aires, Argentina. \\{\tt asolotar@dm.uba.ar}

\end{document}